\documentclass[10pt,a4paper,twoside]{article}
\usepackage{amsmath,amsfonts,amsthm,amsopn,amssymb,enumitem}
\usepackage{palatino}
\usepackage{graphicx,float}
\usepackage[font=footnotesize]{caption}%,labelfont=bf
\usepackage[colorlinks=true]{hyperref} \hypersetup{urlcolor=blue,
citecolor=red, linkcolor=blue}

\newcommand{\eps}{\varepsilon}

\newcommand{\R}{\mathbb{R}}
\newcommand{\Z}{\mathbb{Z}}

\newcommand{\RN}{{\mathbb{R}^N}}

\renewcommand{\le}{\leqslant}
\renewcommand{\ge}{\geqslant}
\renewcommand{\a }{\alpha }

\newcommand{\g }{\gamma }

\newcommand{\n }{\nabla }
\newcommand{\s }{\sigma }
\renewcommand{\t}{\theta}
\renewcommand{\O}{\Omega}

\renewcommand{\S}{\Sigma}
\renewcommand{\L}{\Lambda}

\newcommand{\W}{{\cal W}}

\newcommand{\Ne}{\mathcal{N}}

\renewcommand{\S}{\mathcal{S}}
\newcommand{\N}{\mathbb{N}}

\newcommand{\irn }{\int_{\RN}}

\def\bbm[#1]{\mbox{\boldmath $#1$}}
\newcommand{\beq }{\begin{equation}}
\newcommand{\eeq }{\end{equation}}

\newtheorem{theorem}{Theorem}[section]
\newtheorem{lemma}[theorem]{Lemma}

\newtheorem{definition}[theorem]{Definition}
\newtheorem{proposition}[theorem]{Proposition}
\newtheorem{remark}[theorem]{Remark}

\title{{\bf Quasilinear elliptic equations in $\RN$ \\ via variational methods and \\ Orlicz-Sobolev embeddings
\footnote{The authors are supported by M.I.U.R. - P.R.I.N. ``Metodi
variazionali e topologici nello studio di fenomeni non lineari'' and by G.N.A.M.P.A. Project ``Metodi variazionali e problemi ellittici non lineari"}}}

\author{A. Azzollini \thanks{Dipartimento di Matematica ed Informatica, Universit\`a degli
Studi della Basilicata,  Via dell'Ateneo Lucano 10, I-85100 Potenza,
Italy, e-mail: {\tt antonio.azzollini@unibas.it}}
 \; \& \;
P. d'Avenia\thanks{Dipartimento di Matematica, Politecnico di Bari,
Via E. Orabona 4, I-70125 Bari, Italy, e-mail: {\tt
p.davenia@poliba.it}}
 \; \& \;
A. Pomponio\thanks{Dipartimento di Matematica, Politecnico di Bari,
Via E. Orabona 4, I-70125 Bari, Italy, e-mail: {\tt
a.pomponio@poliba.it}}}

\date{}

\begin{document}

\maketitle

\begin{abstract}
In this paper we prove the existence of a nontrivial non-negative radial solution for
the quasilinear elliptic problem
\begin{equation*}\label{eq}
\left\{
\begin{array}{ll}
-\n \cdot \left[\phi'(|\n u|^2)\n u  \right] +|u|^{\a-2}u =|u|^{s-2} u, & x\in
\RN,
\\
u(x) \to 0 , \quad \hbox{as }|x|\to \infty,
\end{array}
\right.
\end{equation*}
where $N\ge 2$, $\phi(t)$
behaves like $t^{q/2}$ for small $t$ and $t^{p/2}$ for large $t$,
 $1< p<q<N$, $1<\a\le p^* q'/p'$ and $\max\{q,\a\}< s<p^*$, being $p^*=\frac{pN}{N-p}$ and $p'$ and $q'$ the conjugate exponents, respectively, of $p$ and $q$.
Our aim is to approach the problem variationally by using the tools
of critical points theory in an Orlicz-Sobolev space.
A multiplicity result is also given.
\end{abstract}

\section{Introduction}

This paper deals with the following quasilinear
elliptic equation
\begin{equation}\label{eqf}
-\n \cdot \left[\phi'(|\n u|^2)\n u  \right] =f(u) \qquad \hbox{in
} \RN,\; N\ge 2,
\end{equation}
where $\phi \in C^1(\R_+,\R_+)$ has a different growth near zero
and infinity. Such a type of  behaviour occurs, for example, when
$\phi(t)=2[(1+t)^{\frac 12}-1]$. In this case (\ref{eqf}) becomes
\[
-\n \cdot\left( \frac{\n u}{\sqrt{1+|\n u|^2}}  \right)=f(u),
\]
known as the prescribed mean curvature equation or the capillary surface equation.

Such a kind of problems has been deeply
studied in the recent years: existence and non-existence results of solutions decaying to zero at infinity have been proved by \cite{BC,CG,DG,FLS,FIN,KS,NiSe, NiSe2,PS}, among others, under different assumptions on the nonlinearity $f$ and on the function $\phi$.
Moreover, for bounded domains, we recall \cite{CGMS,FN07,MT,NaSu}.

More precisely we are interested in the existence of solutions of
the following quasilinear elliptic problem
\begin{equation}\label{eq}\tag{${\cal P}$}
\left\{
\begin{array}{l}
-\n \cdot \left[\phi'(|\n u|^2)\n u  \right] +|u|^{\a-2 }u=|u|^{s-2} u,\qquad x\in \RN,
\\
u(x) \to 0 , \quad \hbox{as }|x|\to + \infty,
\end{array}
\right.
\end{equation}
where $N\ge 2$, $\phi(t)$ behaves like $t^{q/2}$ for small $t$ and $t^{p/2}$ for large $t$,
$1< p<q<N$, $1<\a\le p^*q'/p'$ and $\max\{q,\a\}<s<p^*=\frac{pN}{N-p}$, being $p'$ and $q'$ the conjugate exponents, respectively, of $p$ and $q$.

Our aim is to approach the problem variationally by using the tools
of critical points theory. A non-trivial difficulty, which
immediately appears, consists in identifying the right functional
setting for the problem.
% The different growth at zero and at
%infinity of the principal part and the unboundedness of the domain
%advise us not to use classical Sobolev spaces.
%Indeed, solutions of \eqref{eq} are, at least formally, the critical points of the functional
%\[
%I(u)=\frac 12 \irn \phi (|\n u|^2)
%+\frac 1\a \irn |u|^\a
%-\frac 1s \irn |u|^s,
%\]
%where
%\[
%\phi (|\n u|^2)\simeq\left\{
%\begin{array}{ll}
%\displaystyle{ |\n u|^p,}
%& \hbox{if }|\n u|\gg 1,
%\\
%\displaystyle{ |\n u|^q,}
%& \hbox{if }|\n u|\ll 1.
%\end{array}
%\right.
%\]
%Since no inclusion relation holds between $W^{1,p}(\RN)$ and $W^{1,q}(\RN)$,
%we need to introduce  a new functional framework.
%
%
%------------------
Solutions of \eqref{eq} are, at least formally, the critical points of the functional
\[
I(u)=\frac 12 \irn \phi (|\n u|^2)
+\frac 1\a \irn |u|^\a
-\frac 1s \irn |u|^s,
\]
where
\[
\phi (|\n u|^2)\simeq\left\{
\begin{array}{ll}
\displaystyle{ |\n u|^p,}
& \hbox{if }|\n u|\gg 1,
\\
\displaystyle{ |\n u|^q,}
& \hbox{if }|\n u|\ll 1.
\end{array}
\right.
\]
This different growth at zero and at
infinity of the principal part and the unboundedness of the domain
advise us not to use classical Sobolev spaces and to introduce  a new functional framework.
So, we define a sort of Orlicz-Sobolev space with respect to which
the functional is well defined and $C^1$. In
this direction, a first step is to show, by suitable embedding
theorems, that all the parts of the functional are finite and
controlled by the norm of our space. In particular, since there are
power-like nonlinearities, we need to study the embedding of our
space into a Lebesgue ones. At this stage, we look at the results
obtained in \cite{BPR,BF,DS} on the sum of Lebesgue spaces to recover
some useful known properties on our space and prove a fundamental
continuous
embedding theorem.\\
Afterwards we deal with the compactness properties of the
functional. Actually, as in the situation of semilinear elliptic
equations, the main difficulty to get compactness lies in the fact
that in unbounded domains the group of translations constitutes an
obstruction to compact embeddings. To overcome this difficulty, as
in \cite{BL,Str}, we need to constrain the functional to a suitable
space which is not invariant with respect to the translations. In
view of this, we restrict the domain of the functional to the
Orlicz-Sobolev space obtained by density starting from radially
symmetric test functions. Proceeding in analogy with the well known
result due to Strauss \cite{Str}, we are able to get uniformly
decaying estimates that we use to show that our space compactly
embeds into certain Lebesgue spaces. As a consequence it is easy to
check that our functional satisfies the Palais-Smale condition, a
first step in view of the application of the Mountain Pass Theorem.
\\
%However, first we have to establish the relation between the
%critical points of the functional and the solutions of our problem.
%Unfortunately, it is at once seen that a direct application of the
%Palais' symmetric criticality principle is not immediately available
%because of the intrinsic difficulty in making computations with the
%norm of our space. We bypass this difficulty by performing a new
%version of the Palais' result known as {\bf Azzollini-Palais
%Theorem}.\\
At this point, we test the geometrical assumptions of Mountain Pass Theorem in this setting and get our goal.\\

In order to state more precisely our results, let
$1< p<q$ and $\phi \in
C^1(\R_+,\R_+)$ be such that
\begin{enumerate}[label=(${\rm \Phi}$\arabic*), ref=${\rm \Phi}$\arabic*]
\item \label{phi1} $\phi(0)=0$;
\item \label{phi2} there exists a positive constant $c$ such that
\[
\left\{
\begin{array}{ll}
c t^{\frac p2} \le \phi(t), & \hbox{if }t\ge 1,
\\
c t^{\frac q2} \le \phi(t), & \hbox{if }0\le t\le 1;
\end{array}
\right.
\]
\item \label{phi3} there exists a positive constant $C$ such that
\[
\left\{
\begin{array}{ll}
\phi(t) \le C t^{\frac p2}, & \hbox{if }t\ge 1,
\\
\phi(t) \le C t^{\frac q2}, & \hbox{if }0\le t\le 1;
\end{array}
\right.
\]
\item \label{phi4} there exists $0<\mu<1$ such that
\[
\phi'(t)t\le \frac{s\mu}{2}\phi(t), \quad \hbox{for all }t\ge 0;
\]
\item \label{phi5} the map $t \mapsto \phi(t^2)$ is strictly convex.
\end{enumerate}

\begin{remark}
Observe that by \eqref{phi4} and \eqref{phi5}, we infer that
\[
\phi(t)<2\phi'(t)t\le s\mu\phi(t), \quad \hbox{for all }t> 0,
\]
and so $s\mu>1$.
\end{remark}

\begin{remark}
As an example, a function that satisfies all the previous
assumptions is
\[
\phi(t)=\frac{2}{p} \left[(1+t^\frac{q}{2})^\frac{p}{q}-1\right].
\]
In this case the problem (\ref{eq}) becomes
\begin{equation}\label{eq:es}
\left\{
\begin{array}{ll}
-\n \cdot \left[(1+|\n u|^q)^{\frac pq -1}|\n u|^{q-2}\n u  \right]
+|u|^{\a-2}u =|u|^{s-2} u, \qquad x\in \RN,
\\
u(x) \to 0 , \quad \hbox{as }|x|\to +\infty.
\end{array}
\right.
\end{equation}
Let us observe that even if Franchi, Lanconelli \& Serrin in \cite{FLS} treat a very general quasilinear equation, Theorem A does not apply to  problem \eqref{eq:es} since, with their notation, in this case we have $H(\infty)=\O(\infty)=F(\g)=\infty$ but $\displaystyle\lim_{u \to \infty} \frac{H^{-1}(F(u))}{u}=\infty$.
\end{remark}

Now we state our main results.

\begin{theorem}\label{main}
Assuming that $1< p<q<\min\{N,p^*\}$,  $1<\a\le p^*q'/p'$,
$\max\{q,\a\}<s<p^*$ and (\ref{phi1}-\ref{phi5}), there exists a
nontrivial non-negative radially symmetric solution of (\ref{eq}).
\end{theorem}

\begin{remark}
By well known results by Pucci, Serrin \& Zou \cite{PSZ,SZ}, if $\a \ge q$, we infer that the solution found is positive; on the contrary it has compact support, if $\a<q$ and requiring in addiction that there exists  $1<\nu\le s\mu$ such that
\[
\frac{\nu}{2}\phi(t) \le  \phi'(t)t,
\]
for $t$ sufficiently small.
\end{remark}

\begin{theorem}\label{infty}
Assuming that $1< p<q<\min\{N,p^*\}$,  $1<\a\le p^*q'/p'$,
$\max\{q,\a\}<s<p^*$ and (\ref{phi1}-\ref{phi5}), there exist
infinitely many radially symmetric solutions of (\ref{eq}).
\end{theorem}

\begin{remark}
Any couple $(p,q)$ in the interior of the coloured region in Figure \ref{fig:f} is admissible for our problem.
\begin{figure}[H]
\centering
\includegraphics[height=5cm]{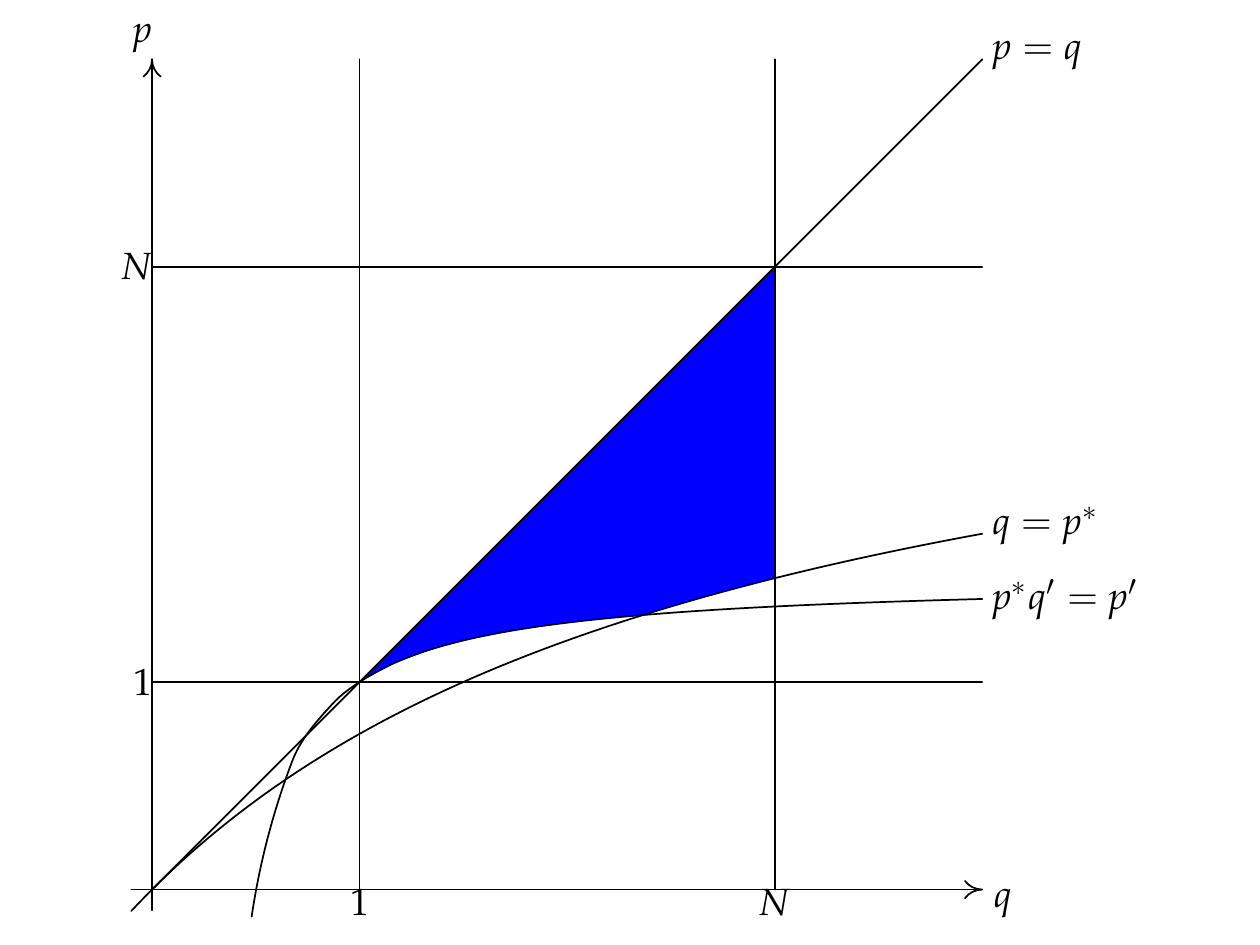}
\\
{\footnotesize \caption{Admissible $(p,q)$.}\label{fig:f}}
\end{figure}
\end{remark}

\begin{remark}
Theorem \ref{main} holds also if in the right hand side of (\ref{eq}), instead of a pure power nonlinearity, we consider a more general one which satisfies the Ambrosetti-Rabinowitz growth condition. More precisely, using slightly modified arguments, we can treat the following problem
\begin{equation*}
\left\{
\begin{array}{ll}
-\n \cdot \left[\phi'(|\n u|^2)\n u  \right] +|u|^{\a-2}u =f(u), & x\in \RN,
\\
u(x) \to 0 , \quad \hbox{as }|x|\to \infty,
\end{array}
\right.
\end{equation*}
where $f \in C(\R,\R)$ satisfies
\begin{enumerate}[label=({\rm f\arabic*}), ref={\rm f\arabic*}]
\item $f(t)=o(t^{\a-1})$, as $t\to 0^+$,%\label{f1}
\item $f(t)=o(t^{p^*-1})$, as $t\to +\infty$,%\label{f2}
\item if $F(t)=\int_0^t f(z)dz$, there exists $\t>\a$ such that%\label{f3}
\[
0<\t F(t)\le f(t)t, \quad \hbox{for all }t> 0,
\]
%\begin{enumerate}[label=({\rm f\arabic*}), ref={\rm f\arabic*}]
%\setcounter{enumi}{4}
\item  $\displaystyle \liminf_{t\to+\infty}\frac{f(t)}{t^{q-1}}>0$, if $\a<q$.
%\end{enumerate}
\end{enumerate}
and $\phi \in C^1(\R_+,\R_+)$ satisfies  (\ref{phi1}), (\ref{phi2}), (\ref{phi3}), (\ref{phi5}) and
\begin{enumerate}[label=(${\rm \Phi\arabic*'}$), ref=${\rm \Phi\arabic*'}$]
\setcounter{enumi}{3}
\item there exists $0<\mu<1$ such that%\label{phi4'}
\[
2\phi'(t)t \le \t\mu\phi(t), \quad \hbox{for all }t\ge 0.
\]
\end{enumerate}
Theorem \ref{infty} holds requiring also that
\begin{enumerate}[label=({\rm f\arabic*}), ref={\rm f\arabic*}]
\setcounter{enumi}{4}
\item $f$ is odd.%\label{f4}
\end{enumerate}
\end{remark}

The paper is organized in the following way.\\
In Section \ref{Antoniofigo} we introduce the functional framework
and list some fundamental properties of the space. In particular in
this part we study the relation between our space and the classical
Lebesgue spaces and provide new continuous and compact embedding
theorems.\\
In Section
\ref{SefossiunafemminabonaiomitrombereiAntonioinunfienile} we verify
that the functional has a good geometry and compactness to apply both the classical Mountain Pass Theorem and its $\Z_2$-symmetric version. We also show that, strengthening a little bit our assumptions, we are able to prove the existence of a ground state solution in the set of all the radially symmetric solutions.

\subsection*{Notation}
\begin{itemize}

\item $\mathbb{K}=\R$ or $\mathbb{K}=\RN$ according to the case.
\item If $r>0$, we denote by $B_r$ the ball of center $0$ and radius $r$.
\item If $\O\subset \RN$, then $\O^c=\RN \setminus \O$.
\item Everytime we consider a subset of $\RN,$ we assume it is measurable and we denote by $|\cdot|$ its measure.
\item If $\O\subset\RN$, $\tau \ge 1$ and $m\in\N^*$,  we denote by $L^\tau (\O)$ the Lebesgue space $L^\tau (\O,\mathbb{K})$, by $\|\cdot\|_{L^\tau(\O)}$ its norm ($\|\cdot\|_\tau$ if $\O=\RN$) and by $W^{m,\tau}(\O)$ the usual Sobolev spaces.
%\item If $\tau > 1$, $\tau^*=\frac{N\tau}{N-\tau}$ and $\tau'=\frac{\tau}{\tau-1}$ is the conjugate exponent of $\tau$.

\end{itemize}

\section{The functional setting}\label{Antoniofigo}
This section is devoted to the construction of the functional setting.
\\
As a first step, we have to recall some well known facts on the sum of Lebesgue spaces.
\begin{definition}
Let $1<p<q$ and $\O\subset \RN$.
We denote with $L^p(\O)+L^q(\O)$ the completion of $C_c^\infty (\O,\mathbb{K})$ in the norm
\begin{equation}
\label{eq:pq}
\| u\|_{L^p(\O)+L^q(\O)}=\inf\left\{\| v\|_p+\|  w\|_q\;\vline\; v \in L^p(\O), w\in L^q(\O), u= v+w\right\}.
\end{equation}
We set $\| u\|_{p,q}=\|u\|_{L^p(\RN)+L^q(\RN)}$.
\end{definition}

The spaces $L^p(\O)+L^q(\O)$ are extensively studied in \cite[Section 2]{BPR}, where a slightly different definition is given, but it can be easily shown that the two definitions are equivalent. Moreover, in \cite{BPR}, it has been shown that $L^p(\O)+L^q(\O)$ are Orlicz spaces.

In the next proposition we give a list of properties that will be useful in the rest of the paper.

\begin{proposition}\label{pr:lplq}
Let $\O\subset\RN$, $u\in L^p(\O)+L^q(\O)$ and $\Lambda_{u}=\left\{ x\in\O \; \vline \; | u(x)| > 1 \right\}$. We have:
\begin{enumerate}[label=({\rm \roman*}), ref=({\rm \roman*})]
\item \label{en:op} if $\O'\subset\O$ is such that $|\O'|<+\infty$, then $u\in L^p(\O')$;
\item \label{en:oinf} if $\O'\subset\O$ is such that $u\in L^{\infty}(\O')$, then $u\in L^q(\O')$;
\item \label{en:lu} $|\Lambda_{u}|<+\infty$;
\item \label{en:ulu} ${u} \in L^p(\Lambda_{u}) \cap L^q(\Lambda^c_{u})$;
\item \label{en:min} the infimum in (\ref{eq:pq}) is attained;
\item \label{en:dual} $L^p(\O)+L^q(\O)$ is reflexive and $(L^p(\O)+L^q(\O))'=L^{p'}(\O)\cap L^{q'}(\O)$;
\item \label{en:lupq} $\| u \|_{L^p(\O)+L^q(\O)} \le \max \left\{\| u\|_{L^p(\Lambda_{\bf u})}, \| u\|_{L^q(\Lambda^c_{\bf u})} \right\}$;
\item \label{en:ooc} if $B\subset\O$, then  $\| u\|_{L^p(\O)+L^q(\O)}\le \| u\|_{L^p(B)+L^q(B)}+\| u\|_{L^p(\O\setminus B)+L^q(\O\setminus B)}$.
%if $\O\subset\RN$, then  $\|{\bf u}\|_{p,q}\le \|{\bf u}\|_{L^p+L^q(\O)}+\|{\bf u}\|_{L^p+L^q(\O^c)}$.
\end{enumerate}
\end{proposition}

\begin{proof}
For the proof of properties \ref{en:op}-\ref{en:lupq} we refer to \cite[Section 2]{BPR}. Here we give only the proof of \ref{en:ooc}.\\
Let $ u \in L^p(\O)+L^q(\O)$. Obviously $ u|_B\in L^p(B)+L^q(B)$ and $ u|_{\O\setminus B} \in L^p(\O\setminus B)+L^q(\O\setminus B)$.
So, by \ref{en:min}, we can consider $ v_1\in L^p(B),  v_2\in L^p(\O\setminus B),  w_1\in L^q(B)$ and $ w_2\in L^q(\O\setminus B)$ such that
\[
\begin{array}{lll}
u = v_1 +w_1 \;\hbox{ on }B,&&\|u\|_{L^p(B)+L^q(B)}=\|v_1\|_{L^p(B)}+\|w_1\|_{L^q(B)},\\
u = v_2 +w_2 \;\hbox{ on }\O\setminus B, && \|u\|_{L^p(\O\setminus B)+L^q(\O\setminus B)}=\|v_2\|_{L^p(\O\setminus B)} +\|w_2\|_{L^q(\O\setminus B)}.
\end{array}
\]
Then, if
\[
v=
\left\{
\begin{array}{ll}
v_1&\hbox{ in } B\\
v_2&\hbox{ in } \O\setminus B
\end{array}
\right.
\hbox{ and }
w=
\left\{
\begin{array}{ll}
w_1&\hbox{ in } B\\
w_2&\hbox{ in } \O\setminus B
\end{array}
\right.
\]
we have that $v\in L^p(\O)$, $w\in L^q(\O)$, $u= v+w$ and
\begin{align*}
\|u\|_{L^p(\O)+L^q(\O)} \le & \|v\|_{L^p(\O)} + \|w\|_{L^q(\O)}\\
\le &\|v_1\|_{L^p(B)} + \|v_2\|_{L^p(\O\setminus B)} + \|w_1\|_{L^q(B)} + \|w_2\|_{L^q(\O\setminus B)}\\
= & \|u\|_{L^p(B)+L^q(B)}+\|u\|_{L^p(\O\setminus B)+L^q(\O\setminus B)}.
\end{align*}
\end{proof}

We can now define the Orlicz-Sobolev space where we will study our problem.

\begin{definition}
Let and $\a >1$. We denote with $\mathcal{W}$ the completion of $C_c^\infty (\RN,\R)$ in the norm
\[
\|u\|=\|u\|_\a+\|\n u\|_{p,q}.
\]
\end{definition}

Let us now study some properties of the space $\W$.

\begin{proposition}
$(\mathcal{W},\|\cdot\|)$ is a Banach space.
\end{proposition}

\begin{proof}
Let $\{ u_n\}_n$ be a Cauchy sequence in $\mathcal{W}$. Then $\{ u_n\}_n$ is a Cauchy sequence in $L^\a(\RN)$ and $\{ \n u_n\}_n$ is a Cauchy sequence in $L^p(\RN)+ L^q(\RN)$. Since $L^\a(\RN)$ is complete, there exists $u \in L^\a(\RN)$ such that $\lim_n u_n = u$ in $L^\a(\RN)$. Since $L^p(\RN)+ L^q(\RN)$ is complete, then there exists ${\bf a} \in L^p(\RN)+ L^q(\RN)$ such that $\lim_n \n u_n = {\bf a}$ in $L^p(\RN)+ L^q(\RN)$.
We want to prove that $\n u = {\bf a}$ in the distributions sense, i.e. that for every $\varphi \in C_c^\infty (\RN)$
\[
\int_\RN u \n \varphi  = - \int_\RN \varphi {\bf a}.
\]
Obviously, for every $\varphi \in C_c^\infty (\RN)$ and for every $n\in\N$
\[
\int_\RN u_n \n \varphi = - \int_\RN \varphi \n u_n.
\]
So it is sufficient to prove that
\[
\lim_n \int_\RN u_n \n \varphi = \int_\RN u \n \varphi
\quad\hbox{ and }\quad
\lim_n \int_\RN \varphi \n u_n=\int_\RN \varphi {\bf a}.
\]
Since $\lim_n u_n = u$ in $L^\a(\RN)$, then
\[
\left|\int_\RN  ( u_n - u)\n\varphi\right| \leq \|\n\varphi\|_{\a'} \|u_n - u\|_\a \to 0.
\]
Moreover, for every $n\in\N$, from \ref{en:min} of Proposition \ref{pr:lplq}, we can consider $({\bf v}_n, {\bf w}_n)\in L^p(\RN)\times L^q(\RN) $ such that
\[
\n u_n - {\bf a} = {\bf v}_n + {\bf w}_n \hbox{ and } \| \n u_n - {\bf a} \|_{p,q} = \|{\bf v}_n\|_p + \|{\bf w}_n\|_q.
\]
Since $\lim_n \| \n u_n - {\bf a} \|_{p,q}=0$, then $\lim_n \|{\bf v}_n\|_p = \lim_n \|{\bf w}_n\|_q =0$. Thus, if $\varphi \in C_c^\infty (\RN)$
\begin{align*}
\left|\int_\RN \varphi (\n u_n - {\bf a})\right| & =  \left| \int_\RN \varphi {\bf v}_n + \int_\RN \varphi {\bf w}_n \right|\\
%\le & \int_\RN |\varphi| |{\bf v}_n| + \int_\RN |\varphi| |{\bf w}_n|\\
& \le  \|\varphi\|_{p'} \|{\bf v}_n\|_{p} + \|\varphi\|_{q'} \|{\bf w}_n\|_{q} \to 0.
\end{align*}
\end{proof}

\begin{proposition}\label{pr:rifle}
$(\mathcal{W},\|\cdot\|)$ is reflexive.
\end{proposition}

\begin{proof}
As in \cite{BPR}, on $L^p(\RN)+L^q(\RN)$ we can also consider  the following norm
\begin{equation*}
%\label{eq:tpq}
\| u\|^*_{p,q}=\inf\left\{\left( \| v\|^2_p+\| w\|^2_q\right)^{1/2}\;\vline\; v\in L^p(\RN),w\in L^q(\RN),u= v+ w\right\}
\end{equation*}
and then, on $\mathcal{W}$, the norm
\[
\|u\|^*=\|u\|_\a+\|\n u\|^*_{p,q}.
\]
It can be easily shown that $\|\cdot\|$ and $\|\cdot\|^*$ are equivalent in $\mathcal{W}$.
%Moreover, for every ${\bf v}\in L^p(\RN,\mathbb{K})$ and ${\bf w}\in L^q(\RN,\mathbb{K})$, we can define
%\[
%\|({\bf v}, {\bf w})\|_{p,q}= \| {\bf v}\|_p+\| {\bf w}\|_q
%\]
%and
%\begin{equation}
%\label{eq:np2}
%\|({\bf v}, {\bf w})\|^*_{p,q}= \left( \| {\bf v}\|^2_p+\| {\bf w}\|^2_q\right)^{1/2},
%\end{equation}
%which are two equivalent norms in $L^p(\RN,\mathbb{K})\times L^q(\RN,\mathbb{K})$.\\
%Hence the norms $\|\cdot\|_{p,q}$ and  $\|\cdot\|^*_{p,q}$ are equivalent (being  the infima of equivalent norms in $L^p(\RN,\mathbb{K})\times L^q(\RN,\mathbb{K})$) and then also
Moreover, by \cite[Proposition 2.6]{BPR}, $\|\cdot\|^*_{p,q}$ is uniformly convex.
%Since the norm $\|(\cdot, \cdot)\|^*_{p,q}$ (see (\ref{eq:np2})) is uniformly convex, then also the norm $\|\cdot\|^*_{p,q}$ (see (\ref{eq:tpq})) is uniformly convex. Indeed if $u_1, u_2 \in \mathcal{W}$ satisfies
%\[
%\|\n u_1\|^*_{p,q}\le 1, \quad \|\n u_2\|^*_{p,q}\le 1, \quad \|\n u_1 - \n u_2\|^*_{p,q}>\varepsilon
%\]
%then, by Proposition \ref{pr:inf}, there exist $({\bf v}_1, {\bf w}_1), ({\bf v}_2, {\bf w}_2)\in L^p(\RN,\RN)\times L^q(\RN,\RN)$ such that
%\[
%\|({\bf v}_1, {\bf w}_1)\|^*_{p,q}=\|\n u_1\|^*_{p,q}\le 1,
%\quad
%\|({\bf v}_2, {\bf w}_2)\|^*_{p,q}=\|\n u_2\|^*_{p,q}\le 1,
%\]
%\[
%\varepsilon<\|\n u_1 - \n u_2\|^*_{p,q} \le \|({\bf v}_1-{\bf v}_2, {\bf w}_1-{\bf w}_2)\|^*_{p,q}.
%\]
%Since $\|(\cdot, \cdot)\|^*_{p,q}$ is uniformly convex, then there exists $\d>0$ such that
%\[
%\frac{1}{2} \|\n u_1 + \n u_2 \|^*_{p,q} \le \frac{1}{2} \| ({\bf v}_1 + {\bf v}_2, {\bf w}_1 + {\bf w}_2) \|^*_{p,q}<1-\d.
%\]
So, on $\mathcal{W}$ we can consider two uniformly convex norms: $\|\n\cdot\|^*_{p,q}$ and the $L^\a(\RN)$ norm. By a well known general result, also the norm
\[
\|\cdot \|^\# =\sqrt{\| \cdot \|_\a^2 + (\| \n \cdot \|^*_{p,q})^2}
\]
is uniformly convex and then $(\mathcal{W}, \|\cdot \|^\#)$ is reflexive. But, since the norm $\|\cdot \|^\#$ is equivalent to $\|\cdot \|$, then, also $(\mathcal{W}, \|\cdot \|)$ is reflexive.
\end{proof}

Adapting some classical arguments (see e.g. \cite{B}) we prove the following embedding result.

\begin{theorem}\label{th:immersione}
If $1<p<\min\{q,N\}$ and $1< p^*\frac{q'}{p'}$ then, for every $\a\in\left(1,p^*\frac{q'}{p'}\right]$,
%\begin{equation}
%\label{eq:hypce}
%%\frac{2N}{N+2} \le q < N
%%\quad
%%\hbox{ and }
%%\quad
%\frac{N(3q-2)}{(N+2)q-2} \le p <
%\end{equation}
the space $\mathcal{W}$ is continuously embedded into $L^{p^*}(\RN)$.
\end{theorem}

\begin{proof}
Let $\varphi\in C_c^\infty (\RN)$ and $t\ge 1$. It can be proved that (see \cite[page 280]{B})
\[
\|\varphi\|_{tN/(N-1)}^t \le t \prod_{i=1}^{N} \left\| |\varphi|^{t-1} \frac{\partial \varphi}{\partial x_i}\right\|_1^{1/N}.
\]
By \ref{en:dual} of Proposition \ref{pr:lplq} we have
\begin{equation}
\label{eq:br2}
\|\varphi\|_{tN/(N-1)}^t \le t \left( \|\varphi\|_{p(t-1)/(p-1)}^{t-1} + \|\varphi\|_{q(t-1)/(q-1)}^{t-1} \right) \prod _{i=1}^{N} \left\|\frac{\partial \varphi}{\partial x_i}\right\|_{p,q}^{1/N}.
\end{equation}
If we take $t$ such that $\frac{tN}{N-1}=\frac{p}{p-1}(t-1)$, the inequality (\ref{eq:br2}) can be written as
\begin{equation}
\label{eq:brez}
\|\varphi\|_{p^*}^t \le t \left( \|\varphi\|_{p^*}^{t-1} + \|\varphi\|_{\bar\a}^{t-1} \right) \prod _{i=1}^{N} \left\|\frac{\partial \varphi}{\partial x_i}\right\|_{p,q}^{1/N}
\end{equation}
with $\bar\a=\frac{Nq(p-1)}{(q-1)(N-p)}=p^*\frac{q'}{p'}$.\\
We notice that, since $p<q$, then $\bar{\a}<p^*$.\\
Moreover, for every $i=1,\ldots,N$,
\begin{align*}
\inf\left\{ \| v_i\|_p+\| w_i\|_q \;\vline\; v_i\in L^p(\RN), w_i \in L^q(\RN),\frac{\partial \varphi}{\partial x_i}= v_i+ w_i\right\} \\
\le
\inf\left\{ \| {\bf v}\|_p+\| {\bf w}\|_q \;\vline\; {\bf v}\in L^p(\RN),{\bf w}\in L^q(\RN),\n \varphi ={\bf v}+{\bf w}\right\}.
\end{align*}
Then
\[
\prod _{i=1}^{N} \left\|\frac{\partial \varphi}{\partial x_i}\right\|_{p,q}^{1/N} \le \|\n\varphi\|_{p,q}.
\]
Since $1<\a \le \bar{\a}<p^*$, then, by interpolation and Young inequalities, we have
\[
\|\varphi\|_{\bar{\a}}^{t-1} \le  C( \| \varphi \|_\a^{t-1} +  \| \varphi \|_{p^*}^{t-1}).
\]
Thus, from (\ref{eq:brez}), we obtain that for every $\varphi \in C_c^\infty (\RN,\R)$
\begin{equation}
\label{eq:brez2}
\|\varphi\|_{p^*}^t \le C \left( \|\varphi\|_{p^*}^{t-1} + \|\varphi\|_\a^{t-1} \right) \|\n\varphi\|_{p,q}.
\end{equation}
We claim that the previous inequality holds for any $u\in\W.$ Indeed, for every $u\in \mathcal{W}$, we can consider a sequence $\{\varphi_n\}_n$ in $C_c^\infty (\RN)$ such that $\lim_n \varphi_n = u$ in $\mathcal{W}$.
Then
\begin{align}
&\lim_n \n \varphi_n = \n u \hbox{ in } L^p(\RN)+L^q(\RN), \nonumber\\
&\lim_n \varphi_n = u \hbox{ in } L^\a(\RN), \label{eq:convqo}\\
&\lim_n \varphi_n = u \hbox{ a.e. in } \RN. \nonumber
\end{align}
Applying \eqref{eq:brez2} to $\varphi=\varphi_n-\varphi_m$ for $n,m\ge 1,$ we deduce that $\{\varphi_n\}_n$ is a Cauchy sequence in $L^{p^*}(\RN)$ and, as a consequence, it converges in $L^{p^*}(\RN)$ to a function $v$. On the other hand, by \eqref{eq:convqo} and the uniqueness of the  limit a.e., we deduce that $v=u$ and then
    \begin{equation*}
        \lim_n \varphi_n=u\hbox{ in } L^{p^*}(\RN).
    \end{equation*}
So, applying \eqref{eq:brez2} to $\{\varphi_n\}_n$ and passing to the limit we deduce our claim.
\\
The continuous embedding $\W\hookrightarrow L^{p^*}(\RN)$ can be deduced reasoning as follows: if $\{u_n\}_n$ is a sequence in $\mathcal{W}$ that converges to $u$ in $\mathcal{W}$, we have that
\[
\| u_n - u \|_{p^*}^t \le C \left( \|u_n - u\|_{p^*}^{t-1} + \|u_n - u\|_\a^{t-1} \right) \|\n u_n - \n u\|_{p,q}
\]
and then $\lim_n u_n =u$ in $L^{p^*}(\RN)$.
\end{proof}

\begin{remark}
By interpolation we have that $\mathcal{W}$ is continuously embedded into $L^\tau(\RN)$ for any $\tau\in [\a,p^*]$.
\end{remark}

Requiring somethings more with the respect to the assumptions of Theorem \ref{th:immersione}, we could have a more precise description of the space $\W$.
\begin{theorem}\label{th:w}
If $1<p<\min\{q,N\}$, $1< p^*\frac{q'}{p'}$ and $q<p^*$, then, for every $\a\in\left(1,p^*\frac{q'}{p'}\right]$, we have that
\[
\W=\{u\in L^\a(\RN)\cap L^{p^*}(\RN)\mid \n u\in L^p(\RN)+L^q(\RN)\}.
\]
\end{theorem}

\begin{proof}
Defining
\[
\widetilde\W=\{u\in L^\a(\RN)\cap L^{p^*}(\RN)\mid \n u\in L^p(\RN)+L^q(\RN)\},
\]
we have to show that $\W=\widetilde\W$. By definition of $\W$ and by Theorem \ref{th:immersione}, we have that $\W\subset \widetilde\W$. Now, let $u\in \widetilde \W$, we have to prove that it can be approximated in the $\W$-norm by smooth functions with compact support. We will follow some ideas of \cite{LL}.
\\
\
\\
As a first step, we prove that $u$ can be approximated in the $\W$-norm by compact support functions. Let $k:\RN\to[0,1]$ be a test function such that $k\equiv 1$ in $|x|\le 1$ and $k\equiv 0$ in $|x|\ge 2.$
For any $M>0$, define $v_M= k_M u $, where $k_M(x)=k(\frac{x}{M}),$ and set $A_M=\{x\in \RN\mid M\le |x|\le 2M\}$.
Certainly $v_M$ has a compact support and it is in $L^\a(\RN)\cap L^{p^*}(\RN)$. Moreover, since $\n v_M= k_M\n u+u\n k_M$, we have that $\n v_M\in L^p(\RN)+L^q(\RN)$ if both the terms of the sum are in $L^p(\RN)+L^q(\RN).$ Since $k_M\in L^{\infty}(\RN)$ and $\n u\in L^p(\RN)+L^q(\RN)$, of course $k_M\n u \in L^p(\RN)+L^q(\RN)$. Since $\n k_M$ vanishes in $A_M^c,$ $|A_M|<+\infty$,  $\n k_M\in L^\infty(A_M)$ and $u\in L^{p^*}(\RN)$, we deduce that also $u\n k_M \in L^p(\RN)+L^q(\RN).$
We have easily that
\begin{equation*}%\label{eq:fg}
\|u-v_M\|^\a_\a\le\int_{B_M^c}|u(x)|^\a\,dx=o_M(1),
\end{equation*}
where $o_M(1)$ denotes vanishing functions as $M\to+\infty.$
Then we have to show that
\begin{equation*}%\label{eq:nfg}
\|\n u-\n v_M\|_{p,q}=o_M(1),
\end{equation*}
too.
Using \ref{en:ulu} of Proposition \ref{pr:lplq},  we deduce that $\n u\in L^p(\Lambda_{\n u})\cap L^q(\Lambda_{\n u}^c)$.
\\
Let us observe that
\begin{equation}\label{eq:nfbc}
\|\n u\|_{ L^p(\Lambda_{\n u}\cap B^c_M)}+\|\n u\|_{ L^q(\Lambda_{\n u}^c\cap B^c_M)}=o_M(1),
\end{equation}
and
\begin{equation}\label{eq:fam}
\|u\|_{ L^{p^*}(A_M)}=o_M(1).
\end{equation}
Since $u\in L^{p^*}(\RN)$, by H\"older inequality, we get that $\n v_M\in L^p(\Lambda_{\n u}\cap B_M^c)\cap L^q(\Lambda_{\n u}^c \cap B_M^c)$ and by (\ref{eq:nfbc}) and \ref{en:ooc} of Proposition \ref{pr:lplq}, we have
\begin{align*}
\|\n u-\n v_M\|_{p,q}
&\le \|\n u-\n v_M\|_{L^p(B_M)+L^q(B_M)}
+ \|\n u-\n v_M\|_{L^p(B_M^c)+L^q(B_M^c)}
\\
&\le \|\n u-\n v_M\|_{L^p(\Lambda_{\n u}\cap B_M^c)+L^q(\Lambda_{\n u}\cap B_M^c)}
\\
&\quad +\|\n u-\n v_M\|_{L^p(\Lambda_{\n u}^c\cap B_M^c)+L^q(\Lambda_{\n u}^c\cap B_M^c)}
\\
&\le \|\n u-\n v_M\|_{L^p(\Lambda_{\n u}\cap B_M^c)}
+\|\n u-\n v_M\|_{L^q(\Lambda_{\n u}^c\cap B_M^c)}
\\
&\le \|\n v_M\|_{L^p(\Lambda_{\n u}\cap B_M^c)}
+\|\n v_M\|_{L^q(\Lambda_{\n u}^c\cap B_M^c)}
+o_M(1)
\\
& \le \|u \n k_M\|_{L^p(A_M)}
+\|u\n k_M\|_{L^q(A_M)}
+o_M(1).
\end{align*}
Since
\[
\|u\n k_M\|_{L^p(A_M)}
\le \frac{C}{M}\|u\|_{L^{p^*}(A_M)} |A_M|^{\frac 1N},
\]
and $|A_M|=O(M^N),$ as $M\to+\infty$, by (\ref{eq:fam}), we have that
\[
\|u\n k_M\|_{L^p(A_M)}=o_M(1).
\]
Analogously, if $p<q<p^*$,
\[
\|u\n k_M\|_{L^q(A_M)}
\le \frac{C}{M}\|u\|_{L^{p^*}(A_M)} |A_M|^{\frac 1q-\frac 1{p^*}}
=o_M(1).
\]
Therefore, we can conclude that $v_M \to u$ in the $\W$-norm, as $M\to +\infty$.
\\
\
\\
As a second step, let us show that $u\in \widetilde \W$ can be approximated by smooth functions.
\\
Let $j:\RN\to\R_+$ be in $C^{\infty}_c(\RN)$ a function inducing a probability measure,
$j_\eps(x)=\eps^{-N}j(\frac{x}{\eps})$ and $u_\eps=u* j_\eps\in C^{\infty}(\RN)$ the convolution product of $u$ with $j_\eps$. Since
$\{j_\eps\}_\eps$ are approximations to the identity,
certainly $u_\eps\to u$ in $L^\a(\RN)$, as $\eps \to 0$.
Moreover if we write $\n u={\bf a}+{\bf b}$, with ${\bf a}\in L^p(\RN)$ and ${\bf b}\in L^q(\RN)$, we have $\n u_\eps=\n u*j_\eps={\bf a}*j_\eps+{\bf b}*j_\eps$, with of course ${\bf a}*j_\eps\in L^p(\RN)$ and ${\bf b}*j_\eps\in L^q(\RN)$. Therefore
\[
\|\n u_\eps -\n u\|_{p,q}
\le \|{\bf a}*j_\eps-{\bf a}\|_{p}
+\|{\bf b}*j_\eps-{\bf b}\|_{q}\to 0.
\]
Hence we can conclude that  $u_\eps \to u$ in the $\W$-norm, as $\eps\to 0$.
\\
\
\\
The conclusion of the proof follows immediately observing that $\{v_M*j_\eps\}_{M,\eps}$ are in
$C^{\infty}_c(\RN)$ and approximate $u$ in the $\W$-norm.
\end{proof}

In order to prove some compactness results, we consider radially symmetric functions of $\W$.
\begin{definition}
Let us denote with
\[
(C_c^\infty (\RN,\R))_{{\rm rad}}=\{u\in  C_c^\infty (\RN,\R)\mid u \hbox{ is radially symmetric}\},
\]
and let $\mathcal{W}_r$ be the completion of $(C_c^\infty (\RN,\R))_{{\rm rad}}$ in the norm $\|\cdot\|$, namely
\[
 \W_r=\overline{(C_c^\infty (\RN,\R))_{{\rm rad}}}^{\|\cdot\|}.
\]
\end{definition}

\begin{remark}\label{re:rad}
In general it is not clear to see if $\W_r$ coincides with the set of radial functions of $\W$. While, if $1<p<\min\{q,N\}$, $1<\a< p^*\frac{q'}{p'}$ and $q<p^*$, then, arguing as in the proof of Theorem \ref{th:w}, we can prove that the two sets are equal.
\end{remark}

The following compact embedding result holds.

\begin{theorem}
\label{th:compact}
If $1<p<q<N$ and $1< p^*\frac{q'}{p'}$ then, for every $\a\in\left(1,p^*\frac{q'}{p'}\right]$,
$\mathcal{W}_r$ is compactly embedded into $L^\tau(\RN)$ with $\a<\tau<p^*$.
\end{theorem}

To show this result we apply \cite[Theorem A.I]{BL}, that we recall here.
\begin{theorem}
\label{th:BL}
Let $P$ and $Q:\R\to\R$ be two continuous functions satisfying
\begin{equation*}%\label{eq:str1}
\lim_{|s|\to +\infty}\frac{P(s)}{Q(s)}=0,
\end{equation*}
$\{v_n\}_n$ be a sequence of measurable functions from $\RN$ to $\R$ such that
\begin{align*}
&\sup_n\irn | Q(v_n)| <+\infty,
\\ %\label{eq:str2}
&P(v_n(x))\to v(x) \ \hbox{ a.e. in }\RN. %\label{eq:str3}
\end{align*}
Then $\|P(v_n)-v\|_{L^1(B)}\to 0$, for any bounded Borel set
$B$.

Moreover, if we have also
\begin{align*}
\lim_{s\to 0}\frac{P(s)}{Q(s)} &=0,\\ %\label{eq:str5}
\lim_{|x|\to+\infty}\sup_n |v_n(x)| &= 0, %\label{eq:str6}
\end{align*}
then $\|P(v_n)-v\|_{1}\to 0.$
\end{theorem}
In order to use the previous result, we need a uniform decaying estimate on the functions of our space. The radial symmetry of the functions allows us to prove the following lemma which is the analogous of the well known result due to Strauss (see \cite{BL} or \cite{Str}).
\begin{lemma}
\label{le:radiallemma}
If $1<p<q<N$, there exists $C>0$ such that for every $u\in \mathcal{W}_r$
\begin{equation}
\label{eq:unifdec}
| u(x) | \le \frac{C}{|x|^\frac{N-q}{q}}\| \n u \|_{p,q}, \quad\hbox{for }|x|\ge 1.
\end{equation}
\end{lemma}

\begin{proof}
Let $u\in (C_c^\infty (\RN))_{\rm rad}$ and ${\bf v}\in L^p(\RN)$ and ${\bf w}\in L^q(\RN)$ such that $\n u={\bf v}+{\bf w}$. Denote by $S^{N-1}$ the boundary of the $N$ dimensional sphere.
If $r\ge 1$,
\begin{align*}
|u(r)| & \le \int_r^{+\infty} |u'(\rho)| d\rho\\
&=  \frac{1}{|S^{N-1}|} \int_{B_r^c} \frac{|\n u |}{|x|^{N-1}}\\
&\le  \frac{1}{|S^{N-1}|} \left( \int_{B_r^c} \frac{| {\bf v} |}{|x|^{N-1}}+  \int_{B_r^c} \frac{| {\bf w} |}{|x|^{N-1}} \right)\\
&\le  \frac{1}{|S^{N-1}|} \left[ \| {\bf v} \|_p  \left(\int_{B_r^c} \frac{1}{|x|^{\frac{(N-1)p}{p-1}}}\right)^\frac{p-1}{p} + \| {\bf w} \|_q \left( \int_{B_r^c} \frac{1}{|x|^{\frac{(N-1)q}{q-1}}} \right)^\frac{q-1}{q}\right]\\
&\le C(N,p,q) \left( \frac{\| {\bf v} \|_p}{r^{\frac{N-p}{p}}} + \frac{\| {\bf w} \|_q}{r^{\frac{N-q}{q}}} \right)\\
&\le   \frac{C(N,p,q)}{r^{\frac{N-q}{q}}}(\| {\bf v} \|_p + \| {\bf w} \|_q).
\end{align*}
Passing to the infimum, we deduce that \eqref{eq:unifdec} holds for any $u\in (C_c^\infty (\RN))_{\rm rad}.$ By the density of $(C_c^\infty (\RN))_{\rm rad}$ and the convergence a.e. in $\RN$, we have that \eqref{eq:unifdec} is true for every $u \in \mathcal{W}_r$.
\end{proof}

\begin{proof}[Proof of Theorem \ref{th:compact}]
Let $\{u_n\}_n$ be a bounded sequence in $\mathcal{W}_r$.
By Lemma \ref{le:radiallemma} we have that $\lim_{|x|\to +\infty} |u_n (x)|=0$ uniformly with respect to $n$.\\
Up to a subsequence $\{u_n\}_n$ converges weakly to a function $u\in \mathcal{W}$.
\\
We prove that $u_n \to u$ a.e. in $\RN$.\\
For every $n\in \N$ and for every $K\subset\RN$ bounded, by \ref{en:op} Proposition \ref{pr:lplq}, $\n u_n \in L^p(K)$ and then $u_n  \in W^{1,\sigma} (K)$, with $\sigma=\min\{\a,p\}$.\\
Moreover,  if ${\bf v}_n \in L^p(\RN)$ and ${\bf w}_n \in L^q(\RN)$ such that $\n u_n= {\bf v}_n + {\bf w}_n$, we have
\begin{multline*}
  \| \n u_n \|_{L^p(K)} \le \| {\bf v}_n \|_{L^p(K)} + \|{\bf w}_n \|_{L^p(K)} \\
  \le \| {\bf v}_n \|_{L^p(K)} + C \|{\bf w}_n \|_{L^q(K)} \le C (\| {\bf v}_n \|_p + \|{\bf w}_n \|_q)
\end{multline*}
and, passing to the infimum,
\[
\| \n u_n \|_{L^p(K)} \le C \| \n u_n \|_{p,q}.
\]
If $|K| \le 1$ the constant $C$ does not depend on $K$ and then $\left\{u_n\right\}_n$ is bounded in $W^{1,\sigma} (K)$. Thus we get that $u_n \to u$ a.e. in $K$. Covering $\RN$ with  sets with measure less than 1, we deduce that $u_n \to u$ a.e. in $\RN$ and $u\in\mathcal{W}_r$.\\
Hence we apply Theorem \ref{th:BL} with $P(t)=|t|^\tau,$  $Q(t)=|t|^\a+|t|^{p^*}$, $v_n=u_n-u$ and $v=0,$ and we get that $\lim_n u_n =u $ (strongly) in $L^\tau(\RN)$.
\end{proof}

\section{Existence and multiplicity of solutions}\label{SefossiunafemminabonaiomitrombereiAntonioinunfienile}

In this section we prove our main existence and multiplicity results. From now on we suppose that all the assumptions of Theorem \ref{main} hold.

Let us define the functional $I: \W \to \R$ as:
\[
I(u)=\frac 12 \irn \phi (|\n u|^2)
+\frac 1\a \irn |u|^\a
-\frac 1s \irn |u|^s.
\]

\begin{proposition}
The functional $I$ is well defined and it is of class $C^1$.
\end{proposition}

\begin{proof}
The conclusion follows easily from, for example, \cite[Lemma 2.2]{DS}.
\end{proof}

We will find solutions of (\ref{eq}) as critical points of the functional $I$.

In the following proposition, we show that the functional $I$ satisfies the geometrical assumptions of the Mountain Pass Theorem. More precisely, we have:

\begin{proposition}\label{pr:MP}
The functional $I$ verifies the following properties:
\begin{enumerate}[label=\roman*)]
\item $I(0)=0$;
\item there exist $\rho,\bar c>0$ such that $I(u)\ge \bar c$, for any $u\in \W$ with $\|u\|= \rho$;
\item there exists $\bar u \in \W$ such that $I(\bar u)<0$.
\end{enumerate}
\end{proposition}

\begin{proof}
Trivially, $I(0)=0$.
\\
Let us check ii).
\\
If $\|u\|$ is sufficiently small, by (\ref{phi2}), \ref{en:lupq} of Proposition \ref{pr:lplq} and since $\W\hookrightarrow L^s(\RN)$, we have that
\begin{align*}
I(u)
& \ge c_1 \int_{\L_{\n u}^c} |\n u|^q
+  c_2 \int_{\L_{\n u}} |\n u|^p
+ \frac 1\a \irn |u|^\a
-\frac 1s \irn |u|^s
\\
& \ge c \max \left( \int_{\L_{\n u}^c} |\n u|^q , \int_{\L_{\n u}} |\n u|^p \right)
+ \frac 1\a \irn |u|^\a
-\frac 1s \irn |u|^s
\\
%& \ge c \min \left(\|\n u\|_{p,q}^q, \|\n u\|_{p,q}^p\right)
%+ \frac 1\a \irn |u|^\a
%-\frac 1s \irn |u|^s
%\\
& \ge c\Big[\|\n u\|_{p,q}^q +\|u\|^\a_{\a} -\|u\|^s_s\Big]
\\
& \ge c \Big[\| u\|^{\max\{\a,q\}}  -\|u\|^s \Big] \ge \bar c.
\end{align*}
Let us check iii).
\\
Let $u\in C^\infty_c(\RN)$, then by (\ref{phi3}), for all $t >0$, we get
\begin{align*}
I(tu)
&\le C_1 \int_{\L_{\n (t u)}^c} |\n (tu)|^q
+  C_2 \int_{\L_{\n (t u)}} |\n (tu)|^p
+ \frac 1\a \irn |tu|^\a
-\frac 1s \irn |t u|^s
\\
&\le C \left[ t^q \irn |\n u|^q
+   t^p\irn |\n u|^p
+ t^\a \irn |u|^\a
-t^s \irn | u|^s \right].
\end{align*}
Therefore, $I(t u)<0$, for a $t$ sufficiently large.
\end{proof}

By Remark \ref{re:rad}, using the standard Palais' result (see \cite{P}), we infer that $\W_r$ is a natural constraint for the functional $I$. So we consider $I$ restricted to this space and we prove that here the Palais-Smale condition holds.

\begin{proposition}\label{pr:PS}
The functional $I|_{\W_r}$ satisfies the Palais-Smale condition.
\end{proposition}

\begin{proof}
Let $\{u_n\}_n\subset \W_r$ be a PS-sequence for the $I$, namely for a suitable $\bar c\in \R$
\[
I(u_n)\to \bar c \quad \hbox{and} \quad I'(u_n)\to 0 \hbox{ in } \W_r'.
\]
Let us show that $\{u_n\}_n$ is bounded. Indeed, by (\ref{phi4}), we have
\begin{align*}
\bar c+o_n(1)\|u_n\|
&= I(u_n) - \frac 1s I'(u_n)[u_n]
\\
&=\irn \left[ \frac 12 \phi(|\n u_n|^2)- \frac 1s \phi'(|\n u_n|^2)|\n u_n|^2 \right]
+\left(\frac 1\a -\frac{1}{s}\right) \irn |u_n|^\a
\\
&\ge  \frac{1-\mu}{2}\irn  \phi(|\n u_n|^2)
+\left(\frac 1\a -\frac{1}{s}\right) \irn |u_n|^\a
\\
&\ge  c \Big[\min \left(\|\n u_n\|_{p,q}^q, \|\n u_n\|_{p,q}^p \right) +\|u_n\|_\a^\a\Big].
\end{align*}
Therefore, by Proposition \ref{pr:rifle} and Theorem \ref{th:compact}, there exists $u_0\in \W_r$ such that
\begin{align}
&u_n\rightharpoonup u_0, \quad \hbox{weakly in }\W_r,\label{eq:debole}
\\
&u_n \to u_0, \quad \hbox{in }L^s(\RN), \label{eq:fortels}
\\
&u_n \to u_0, \quad \hbox{a.e. in }\RN.\nonumber
\end{align}
Inspired by \cite{MR}, we write $I(u)=A(u)-B(u)$, where $A(u)=A_1(u)+A_2(u)$ and
\[
A_1(u)=\frac 12 \irn \phi(|\n u|^2), \quad
A_2(u)=\frac 1\a \irn |u|^\a,\quad
B(u)=\frac 1s \irn |u|^s.
\]
With these notations, we have
\[
A(u_n)-B(u_n)\to \bar c \quad \hbox{and} \quad A'(u_n)-B'(u_n)\to 0 \hbox{ in } \W_r'.
\]
By (\ref{eq:fortels}), we infer that
\[
B(u_n) \to B(u_0) \quad \hbox{and} \quad
B'(u_n)\to B'(u_0)\hbox{ in }\W_r'.
\]
Therefore
\begin{equation}\label{eq:AB}
A'(u_n)\to B'(u_0)\quad\hbox{in }\W_r'.
\end{equation}
Since $A$ is convex, we have
\[
A(u_n)\le A(u_0)+A'(u_n)[u_n-u_0],
\]
and so, by (\ref{eq:debole}) and (\ref{eq:AB}), we get
\[
\limsup_n A(u_n) \le A(u_0).
\]
Moreover, by the weak lower semicontinuity ($A$ is convex and continuous)
\[
A(u_0)\le \liminf_n A(u_n),
\]
and therefore
\begin{equation}\label{eq:AA}
A(u_n)\to A(u_0).
\end{equation}
By (\ref{eq:debole}) and arguing as in \cite[page 208]{LL}, we have
\begin{align}
&\n u_n\rightharpoonup \n u_0, \quad \hbox{weakly in }L^p(\RN)+L^q(\RN), \label{eq:debolelpq}
\\
&u_n\rightharpoonup u_0, \quad \hbox{weakly in }L^\a(\RN), \label{eq:debolel2}
\end{align}
and so by the weak lower semicontinuity ($A_1$ and $A_2$ are convex and continuous)
\begin{align*}
A_1(u_0)\le \liminf_n A_1(u_n),
\\
A_2(u_0)\le \liminf_n A_2(u_n).
\end{align*}
 This, together with (\ref{eq:AA}), implies that
\begin{align}
A_1(u_0)= \lim_n A_1(u_n), \label{eq:A1}
\\
A_2(u_0)= \lim_n A_2(u_n). \label{eq:A2}
\end{align}
By (\ref{eq:debolel2}) and (\ref{eq:A2}), we infer that
\begin{equation*}%\label{eq:fortel2}
u_n \to u_0,   \quad \hbox{in }L^\a(\RN).
\end{equation*}
By (\ref{eq:debolelpq}) and (\ref{eq:A1}) and by \cite[Lemma 2.3]{DS}, we have that
\begin{equation*}%\label{eq:fortelpq}
\n u_n \to \n u_0,   \quad \hbox{in }L^p(\RN)+L^q(\RN).
\end{equation*}
Therefore
\[
u_n \to u_0,   \quad \hbox{in }\W_r
\]
and the proof is concluded.
\end{proof}

\begin{proof}[Proof of Theorem \ref{main}]
The existence of a nontrivial solution follows immediately by Propositions \ref{pr:MP} and \ref{pr:PS}.
\\
%\end{proof}
%
%\begin{proof}[Proof of Theorem \ref{th:pos}]
To find a nontrivial and non-negative solution, we repeat all the previous arguments showing the existence of a nontrivial solution $\bar u\in \W_r$ of the following problem
\begin{equation}\label{eq:troncata}
\left\{
\begin{array}{l}
-\n \cdot \left[\phi'(|\n u|^2)\n u  \right] +|u|^{\a-2}u =g(u), \quad x\in \RN,
\\
u(x) \to 0 , \quad \hbox{as }|x|\to +\infty,
\end{array}
\right.
\end{equation}
where $g:\R\to \R$ is so defined:
\[
g(t)=\left\{
\begin{array}{ll}
t^{s-1},& \hbox{if }t\ge 0,
\\
0,& \hbox{if }t< 0.
\end{array}
\right.
\]
Since $\bar u$ solves (\ref{eq:troncata}) and since, by Theorem \ref{th:w}, $\bar u^-=\min\{\bar u,0\}\in \W_r$, multiplying the equation by $\bar u^-$, we get
\[
\int_{\O^-}\phi'(|\n \bar u |^2)|\n \bar u|^2+|\bar u|^\a=0,
\]
where $\O^-=\{x\in \RN \mid \bar u(x)<0\}$. Since $\phi'(t)\ge 0$, for all $t\ge 0$, we argue that $\bar u\ge 0$ and so it is a nontrivial and non-negative solution of (\ref{eq}).
\end{proof}

\begin{proof}[Proof of Theorem \ref{infty}]
By the $\Z_2$-symmetric version of the Mountain Pass Theorem \cite{AR}, we need only to prove that there exist $\{V_n\}_n$, a sequence of finite dimensional subspaces of $\W_r$ with $\dim V_n=n$, and $\{R_n\}_n$, a sequence of positive numbers, such that $I(u)\le 0$ for all $u\in  V_n\setminus B_{R_n}$.\\
    Consider $\{\varphi_n\}_n$ a sequence of radially symmetric test functions such that, for any $n\ge 1,$ the functions $\varphi_1, \varphi_2,\ldots,\varphi_n$ are linearly independent.
    Denote by $V_n={\rm span}\{\varphi_1, \varphi_2,\ldots,\varphi_n\}\subset (C_c^{\infty}(\RN,\R))_{\rm rad}\subset\W_r.$
\\
By (\ref{phi3}) and since $V_n$ is a finite dimensional space of test functions, we conclude observing that, if $u\in V_n\setminus B_{R_n}$ and $R_n$ is sufficiently large,
\begin{multline*}
I(u)\le C\Big[ \|\n u\|_q^q +\|u\|_\a^\a - \|u\|_s^s\Big]\\
\le  C\Big[ \|u\|^q +\|u\|^\a - \|u\|^s\Big]\le C \Big[ R_n^q +R_n^\a - R_n^s\Big]\le 0.
\end{multline*}
\end{proof}

\subsection{Ground state solution in $\W_r$}

In this section, we will show how, requiring something slightly more on $\phi$, we can find a ground state solution in $\W_r$.

Let us suppose that $\phi$ satisfies:
\begin{enumerate}[label=(${\rm \Phi}$\arabic*$'$), ref=${\rm \Phi}$\arabic*$'$]\setcounter{enumi}{1}
\item \label{phi2'} there exists a positive constant $c$ such that
\[
\left\{
\begin{array}{ll}
c t^{\frac p2-1} \le \phi'(t), & \hbox{if }t\ge 1,
\\
c t^{\frac q2-1} \le \phi'(t), & \hbox{if }0\le t\le 1.
\end{array}
\right.
\]
\end{enumerate}
Of course (\ref{phi2'}) implies (\ref{phi2}).

Let us indicate with $\S$ the set of all nontrivial solutions of (\ref{eq}) in $\W_r$, namely
\[
\S =\{u\in \W_r\setminus \{0\} \mid I'(u)=0\}.
\]
By Theorem \ref{main}, we know that $\S\neq \emptyset$.

The following lemmas hold for the set $\S$.
\begin{lemma}\label{le:norma>c}
There exists a positive constant $\bar c>0$ such that $\|u\|\ge \bar c$, for all $u\in \S$.
\end{lemma}

\begin{proof}
By (\ref{phi2'}) we have
\begin{align*}
\|u\|^s_{s}
&= \irn \phi'(|\n u|^2)|\n u|^2 +\irn |u|^\a
\\
&\ge c \max \left( \int_{\L_{\n u}^c} |\n u|^q , \int_{\L_{\n u}} |\n u|^p \right)
+  \irn |u|^\a
\\
&\ge c \Big[ \|\n u\|^q_{p,q} +\|u\|^\a_{\a}\Big]
\ge c \|u\|^{\max\{\a,q\}}
\ge c \|u\|^{\max\{\a,q\}}_{s}.
\end{align*}
\end{proof}

\begin{lemma}\label{le:I>c}
There exists a positive constant $\bar c>0$ such that $I(u)\ge \bar c$, for all $u\in \S$.
\end{lemma}

\begin{proof}
Let $u \in \S$. Repeating the arguments of the proof of Proposition \ref{pr:PS} and using by Lemma \ref{le:norma>c}, we have
\[
I(u)
= I(u) - \frac 1s I'(u)[u]
\ge  c \Big[\min \left(\|\n u\|_{p,q}^q, \|\n u\|_{p,q}^p \right) +\|u\|_\a^\a\Big]
\ge \bar c.
\]
\end{proof}

\begin{remark}
Let us indicate with $\Ne$ the Nehari manifold associated to the functional $I$, namely
\[
\Ne =\{u\in \W_r\setminus \{0\} \mid I'(u)[u]=0\}.
\]
Then Lemmas \ref{le:norma>c} and \ref{le:I>c} hold also for $\Ne$.
\end{remark}

By Lemma \ref{le:I>c}, we infer that
\[
\s=\inf_{u\in \S}I(u)>0,
\]
and the next theorem shows that this infimum is achieved.

\begin{theorem}
Assuming that $1< p<q<N$, $\max\{q,\a\}<s<p^*$, $1<\a\le p^*q'/p'$ and (\ref{phi1},\ref{phi2'},\ref{phi3}-\ref{phi5}), then in the space $\W_r$, there exists a ground state solution for the problem (\ref{eq}), namely there exists a nontrivial solution $\bar u\in \W_r$,  such that
\[
I(\bar u)= \min_{u\in \S}I(u).
\]
\end{theorem}

\begin{proof}
Let $\{u_n\}_n\subset \S$ be a minimizing sequence, namely
\[
I(u_n)\to \s \quad \hbox{and}\quad I'(u_n)=0.
\]
Then  $\{u_n\}_n$ is a PS-sequence for $I$ and we conclude by means of Proposition \ref{pr:PS}.
\end{proof}

\end{document}